\documentclass[12pt]{amsart}
\advance\hoffset by -0.5 truecm
\usepackage{amssymb}
\usepackage{hyperref}

\newtheorem{Theorem}{Theorem}[section]

\newtheorem{lemma}[Theorem]{Lemma}

\def \R{{\mathbb R}}

\def \0{\lambda_{0}}

\begin{document}

\title[]{Degenerate Solutions of Yamabe-Type  Equations on Products of Spheres}

\author{H\'ector Barrantes G. and Jorge Dávila }

\address{\textsc{H\'ector Barrantes G.} Universidad de Costa Rica. Sede de Occidente. 20201. Alajuela. Costa Rica.}
\email{hector.barrantes@ucr.ac.cr}%

\address{\textsc{Jorge D\'avila Ortiz.}
	Department of Engineering and Science ITESM, Campus Le\'on Guanajuato\\
	Le\'on Guanajuato 37190, M\'exico.} 

\email{jorge.davila@tec.mx}

\begin{abstract}
We study Yamabe-type equations on the product of two spheres $(S^n \times S^n, G_\delta)$, where $G_\delta$ is a family of Riemannian metrics parametrized by $\delta > 0$. Using bifurcation theory and isoparametric functions, we establish the existence of degenerate solutions that are invariant under the diagonal action of $O(n+1)$ and depend non-trivially on both factors. Our analysis relies on the properties of Gegenbauer polynomials and a careful application of local bifurcation techniques for simple eigenvalues. These results extend previous studies by demonstrating the emergence of solutions that do not solely depend on a single factor, thereby providing new insights into the structure of solutions for Yamabe-type problems on product manifolds.
\end{abstract}

\keywords{Yamabe-type equation, degenerate solutions, bifurcation theory, product of spheres, isoparametric functions, Gegenbauer polynomials, elliptic PDEs on manifolds, scalar curvature}

\subjclass[2020]{35J60, 35B32, 58J55, 58E11}
\maketitle

\section{Introduction}

Given a Riemannian manifold $(M^n , g)$ of dimension $n\geq 3$,  a {\it Yamabe-type } equation  is an elliptic equation of the form

\begin{equation}\label{eq1}
	-  \Delta_g u + \lambda u = \lambda u^{q-1}
\end{equation}

\noindent where  $\lambda >0$  and $q >2$.   Equation (\ref{eq1}) is said to be subcritical if $q-1 < p_{n}: =\frac{n+2}{n-2}$, critical if $q -1 =  p_{2n}$  and supercritical if $q  -1>  p_{2n}$.    The critical case  appears when trying to solve the Yamabe problem which consists in finding metrics of constant scalar curvature in the conformal class $[g]$ of $g$.   If $h = u^{p_n-1} g  \in [g]$ where $u \colon  M \to \R_{\geq 0}$ ,then the scalar curvature $s_h$ of $h$ is equal to a constant $\lambda  \in \R$  if and only if $u$ satisfies

\begin{equation}\label{eq2}
	-  a_n \Delta_g u + s_g u = \lambda u^{p_n}
\end{equation}

\noindent Where $a_{n}:= \frac{4(n-1)}{n-2}$.  Note that  eq.  (\ref{eq1})  has the form of  eq.  (\ref{eq2})  with $\lambda = \frac{s_g}{a_n} $
 and $q-1 = p_n$.

 We are interested in the  product manifold  $\mathbb{S}^n \times \mathbb{S}^n$, equipped with metrics of the form $G_\delta = g_0^n + \delta g_0^n$, where $g_0^n$ is the standard round metric on the sphere and $\delta > 0$ is a fixed parameter. 
 In connection with computations of the Yamabe constant in \cite{Akutagawa}, a natural question arose as to whether all energy-minimizing solutions of the Yamabe equation on certain Riemannian products (such as \(\mathbb{S}^n \times \mathbb{S}^n\) or \(\mathbb{S}^n \times \mathbb{R}^n\)) depend only on one of the factors. Invariant solutions under the diagonal action of \(O(n+1)\) have been shown to exhibit remarkable structure, and their existence has been established via variational and bifurcation methods \cite{Petean1,Petean-Barrantes}. Recent advances \cite{Petean-Betancourt-Batalla} have uncovered new families of such solutions.

 The appearance of \emph{degenerate solutions}—that is, solutions for which the linearized operator possesses a nontrivial kernel—is an important and intriguing aspect in this context. Such solutions typically indicate the presence of a bifurcation point from which new branches of nontrivial solutions emerge, and they are closely related to qualitative changes in the structure of the solution set. Recent developments concerning degenerate solutions of Yamabe-type equations on the standard sphere serve as a foundation for the present work.

 In \cite{Petean2}, Petean constructed nontrivial degenerate solutions by analyzing bifurcation branches associated with isoparametric functions on the sphere. These solutions emerge from bifurcation points of the trivial branch and exhibit symmetry properties determined by the level sets of Cartan--Münzner polynomials. More recently, Catalán and Petean \cite{Catalan-Petean} refined this theory by computing the local behavior of bifurcating branches and establishing explicit criteria under which bifurcation points are transcritical. 
 
 In this work, we construct explicit families of degenerate solutions to equation \eqref{eq1} on $(\mathbb{S}^n \times \mathbb{S}^n, G_\delta)$, which are invariant under the diagonal $O(n+1)$-action and exhibit nontrivial dependence on both factors. Our approach relies on the identification of a natural isoparametric function $f(p,q) = \langle p, q \rangle$, which induces a cohomogeneity-one action on the manifold and permits a reduction of the problem to an ordinary differential equation involving Gegenbauer polynomials.
 
 Applying the Crandall--Rabinowitz bifurcation theorem \cite{Crandall-Rabinowitz}, we establish the existence of bifurcation points $\lambda_k$ at which branches of nontrivial solutions emerge. For each even integer $k$, we show that these branches contain degenerate solutions with precisely $k$ nodal domains. This construction yields a new class of solutions that both complements and extends the families described in \cite{Petean-Betancourt-Batalla}, offering new insight into the solution structure of Yamabe-type problems on product manifolds.
 
 The paper is structured as follows. In Section 2, we review the geometric framework and introduce the isoparametric reduction. Section 3 is devoted to local bifurcation analysis and the existence of degenerate solutions. In Section 4, we conclude with the proof of the main result.

 \vspace{0.5cm}
We consider  $n\geq 2$ and let $g_0^n$ denote the curvature 1 metric on $\mathbb{S}^n$. For any $\delta >0$ we consider the  product  $\mathbb{S}^n \times \mathbb{S}^n $ with the Riemannian metric  $ G_{\delta}:=g_0^n +\delta g_0^n$. Then eq.  (\ref{eq1})   has the form

\begin{equation}\label{eq3}
	-\Delta_{G_{\delta}} u + \lambda u = \lambda u^{q -1} 
\end{equation}

\noindent If  $\lambda  = \frac{n(n-1) (1+\delta^{-1} ) }{a_{2n}}$ and $q-1 = p_{2n}  = \frac{2n+2}{2n-2}$ then  the equation  (\ref{eq3}) is the Yamabe equation on  $ (\mathbb{S}^n \times \mathbb{S}^n ,  G_{\delta})$. Then  the equation (\ref{eq3}) is subcritical if $q-1 < p_{2n}$ , critical if $q -1=  p_{2n}$  and supercritical if $q >  p_{2n}$.

We also  consider the isometric action of $O(n+1)$ on $(\mathbb{S}^n \times \mathbb{S}^n , G_{\delta} )$ given by $$A\cdot(x,y)=(Ax,Ay)$$
 and the function $f \colon \mathbb{S}^n \times \mathbb{S}^n \to [-1,1]$  given by
$$f(p, q) = \langle p, q \rangle.$$

\noindent Note that $f$ is invariant by the action of $O(n+1)$.  By a direct computation we obtain

$$\Delta_{G_{\delta}}f = -n\left( 1+ \frac{1}{\delta}\right) f , \;\;\;\;\;\;
\left|\nabla_{G_\delta}f\right|_{G_\delta}^2 =\left( 1+ \frac{1}{\delta}\right)(1-f^2).$$

This implies that $f$ is  an isoparametric function (see \cite{Wang} for the definition and basic results concerning isoparametric functions). The only critical values of $f$ are its minimum -1 and its maximum 1. The focal manifolds are  $f^{-1}(-1) := M_1,  f^{1}(1) := M_2$ which are diffeomorphic to $\mathbb{S}^n$. 
\noindent  Let $q_f := \frac{2n - dim M_i +2}{2n - dim M_i -2 }  = \frac{n+2}{n-2}$  and note that $p_{2n} < q_f$

 Using the theory of local bifurcation for simple eigenvalues of  M. G. Crandall, P. H. Rabinowitz \cite{Crandall-Rabinowitz}, we have the following result

\begin{Theorem}\label{t1}
  For any $q \in (2, q_f)$ and any $\delta >0$,  let $\lambda_k:= \lambda_{k,\delta, q} := \frac{k(k+n-1)}{q-2}\left(1+ \frac{1}{\delta}\right)$.
  Then for any positive integer $k$and any  $\lambda \in \bigl(\lambda_{k }, \lambda_{k+1}\bigr]$   the eq. (\ref{eq3})  has at least $k$  positive solutions invariant  by  the  diagonal action of  $O(n+1)$, which depends non-trivially on both factors.
\end{Theorem}

It is important to emphasize that Theorem \ref{t1} plays a fundamental role in this paper, as it explicitly establishes the existence of positive solutions to equation \eqref{eq3}, invariant under the diagonal action of  $O(n+1)$, 
 which depend non-trivially on both factors of the product manifold $\mathbb{S}^n \times \mathbb{S}^n$. We will not include its proof here, since the existence of such solutions is a direct consequence of Theorem 1.3 established by Petean et al. in   \cite[Theorem 1.3]{Petean-Betancourt-Batalla}.  We refer the interested reader to this reference for the complete details of the proof.

We call a \emph{degenerate solution}  $(u_0, \lambda_0)$ of equation (\ref{eq1}) if the linearized equation
on  $(u_0, \lambda_0)$

$$-\Delta_g v   + \lambda (1-(q-1)u^{q-2})v=0 $$

has a non-trivial solution.

In this article we show that  there exist a degenerate solutions of (\ref{eq3}), invariant by the by the cohomogeneity one action of  $O(n+1)$ on $\mathbb{S}^n \times \mathbb{S}^n$, which depend non-trivially on both factors.

Every invariant function $u\colon \mathbb{S}^n \times \mathbb{S}^n \to \R$ can be written as $u = \varphi \circ f$, where $\varphi \colon [-1,1] \to \R$.
Since $f$ is smooth the regularity of $u$ is equal to the regularity of $\varphi$. We have that

$$\Delta_{G_\delta}u = ( \varphi '' \circ f )  |\nabla_{G_{\delta}} f |^2_{G_{\delta}} + (\varphi ' \circ f)   \Delta_{G_{\delta}} f
=  \left( 1+ \frac{1}{\delta}\right)(1-t^2)  \   \varphi ''(t) -n\left( 1+ \frac{1}{\delta}\right) t  \  \varphi '(t) $$

\noindent for $t\in [-1,1]$.

\noindent Therefore $u$ solves the equation (\ref{eq2})  if and only if
\begin{equation} \label{t}
	(1 - t^2)\varphi''(t) - n t \varphi'(t) + \frac{\lambda(q - 2)}{1 + \delta^{-1}} \left( \varphi^{q - 1}(t) - \varphi(t) \right) = 0
\end{equation}

\noindent for $t\in [-1,1]$.


\begin{Theorem}\label{t2}
 For any positive even  integer $k$ there exist a degenerate solution  (for some $\lambda > 0$) of the  equation \ref{eq3}    on $(\mathbb{S}^n \times \mathbb{S}^n, G_{\delta}) $, invariant  by  the  diagonal action of  $O(n+1)$, which depends non-trivially on both factors.
\end{Theorem}

\textit{Remark:} Since  $q_f =  \frac{n+2}{n-2} > \frac{4n}{2n-2} =  p_{2n} $ the Theorems \ref{t1} and \ref{t2} applies to supercritical equations.

Although the existence of degenerate solutions invariant under the diagonal action of \( O(n+1) \) has been previously established by Petean et al. in a general context (see \cite[Theorem 1.3]{Petean-Betancourt-Batalla}), we present here a complete proof of Theorem 1.2. Our purpose is to explicitly clarify the adaptation of the bifurcation techniques to our specific setting on the product manifold \( (\mathbb{S}^n\times \mathbb{S}^n, G_{\delta}) \), emphasizing the technical details and particular features arising in this scenario.


\section{Local bifurcation theory}
 We denote by $X_I$ the set of  functions on  $\mathbb{S}^n\times \mathbb{S}^n$ invariant  under  the action of $O(n+1)$.
Consider the Banach space

$$C^{2, \alpha}\left(X_I \right):= X_I \cap C^{2, \alpha} \left(\mathbb{S}^n \times \mathbb{S}^n\right) .$$

\noindent Similarly  we  define

 $$C^{0, \alpha}\left(X_I \right):= X_I \cap C^{0, \alpha} \left(\mathbb{S}^n \times \mathbb{S}^n\right) ,$$

\noindent  We identify $C^{2, \alpha}\left(X_I \right)$ with the set of functions 
$\varphi \in C^{2,\alpha}([-1 , 1])$ such
that $\varphi'(-1)=\varphi'(1)=0$. Similarly  the spaces $C^{0, \alpha}\left(X_I \right):= X_I \cap C^{0, \alpha} \left(\mathbb{S}^n \times \mathbb{S}^n\right) $,
 is identified with $C^{0,\alpha}([-1, 1])$.

Given a solution  $u \colon \mathbb{S}^n\times \mathbb{S}^n \to \R_{>0}$ of equation (\ref{eq3})
we let  $ w = u-1$.  Then, $u$ is a solution of  (\ref{eq3}) if and only if $w$ verifies
\begin{equation}\label{eq-w+1}
-\Delta_{G_{\delta}} w  + \lambda (w+ 1) = \lambda (w+1)^{q-1}
\end{equation}

Assume now that \(w = \varphi \circ f\), where \(\varphi \colon [-1,1] \to \mathbb{R}\) is a smooth function and \(f \colon \mathbb{S}^n \times \mathbb{S}^n \to [-1,1]\) denotes the isoparametric function associated with the diagonal action of the orthogonal group \(O(n+1)\). Note that the function \(w\) depends solely on the variable \(t = f(x,y)\), and the equation for \(w\) reduces to the following second-order ordinary differential equation
\begin{equation} \label{eq-4}
	(1 - t^2)\varphi''(t) - n t \varphi'(t) + \frac{\lambda(q - 2)}{1 + \delta^{-1}} \left[(\varphi(t) + 1)^{q - 1} - \varphi(t) - 1 \right] = 0.
\end{equation}

In order for the corresponding function \(u = \varphi \circ f + 1\) to define a smooth solution of the original problem on the entire product manifold \(\mathbb{S}^n \times \mathbb{S}^n\), it is necessary that \(\varphi\) satisfy appropriate boundary conditions at the endpoints \(t = \pm 1\), which correspond to the focal submanifolds (i.e., singular level sets) of the isoparametric function $f$. 
\begin{equation} \label{boundary-cond-1}
	n\, \varphi'(-1) - \lambda \left[ (\varphi(-1) + 1) - (\varphi(-1) + 1)^q \right] = 0,
\end{equation}
\begin{equation} \label{boundary-cond-2}
	-n\, \varphi'(1) - \lambda \left[ (\varphi(1) + 1) - (\varphi(1) + 1)^q \right] = 0.
\end{equation}

We define $F:C^{2, \alpha}\left(X_I \right) \times \R_{\geq 0} \rightarrow C^{0, \alpha}\left(X_I \right)$
by

$$F(w,\lambda )= -\Delta_{G_{\delta}} w +\lambda \left(w+1 -(w+1)^{q-1} \right).$$

We have for any $\lambda \geq 0$ that $F(0,\lambda )=0$ and we will study solutions of $F(u,\lambda )=0$
which bifurcate for the curve $(0,\lambda )$. The local bifurcation theory is well known, any detail about  what we
will use in this section can be found for instance in Chapter 2 of \cite{Ambrosetti-Malchiodi} or in \cite{Nirenberg}.

Note that

$$F_w' (0,\lambda) [v] = -\Delta_{G_{\delta}}v - \lambda (q-2) v .$$

\noindent We write $v(x)=\varphi \circ f $ for a function
$\varphi :[-1, 1 ] \rightarrow \R$ with $\varphi'(-1)=\varphi'(1 )=0$, 
and $F_w ' ( 0,\lambda) [v]=0$ if and only if

\begin{equation}  \label{eq-eigenvalues}
(1-t^2)\varphi''(t)-nt \varphi'(t) +  \frac{\lambda (q-2)}{1+\delta^{-1}} \varphi(t)=0
\end{equation}

For every $\lambda_k:= \frac{k(k+n+1)}{q-2}(1+\delta^{-1})  $, this equation corresponds to the eigenvalue equation for the Laplacian on the sphere and is satisfied by 
the Gegenbauer polynomials  $P_{k,  n}$  which are given by the formula

\begin{equation}\label{P_kn-formula}
P_{k, n} (t) =   (1-t^2)^{-\frac{n-2}{2}}   \frac{(-1)^k}{2^k k!}\frac{d^k}{dt^k}(1-t^2)^{k+ \frac{n-2}{2}} 
\end{equation}

This implies that   $\ker ( F_u ' (0,\lambda_k ))=\langle P_{k, n} \rangle $ has dimension 1.

The Gegenbauer polynomials are particular cases of the Jacobi polynomials (\cite[Lecture 2]{Askey})

$$
P_k^{(\alpha, \beta)}(t) = (1 - t)^{-\alpha} (1 + t)^{-\beta} \frac{(-1)^k}{2^k k!} \frac{d^k}{dt^k} \left[ (1 - t)^{k+\alpha} (1 + t)^{k+\beta} \right]
$$

\noindent Which satisfies the Jacobi equation  
 
 \[
 (1 - t^2)y''(t) + (\beta - \alpha - (\alpha + \beta + 2)t)y'(t) + k(k + \alpha + \beta + 1)y(t) = 0.
 \]
 
\noindent where  $k$ is a positive integer, and $\alpha = \beta > -1$  are real numbers.  
Thus, the polynomials  $P_{k, n} $    are obtained by setting  $ \alpha =  \beta   = \frac{n-2}{2}  $  in the previous equation. So in our  notation $ P_{k, n} := P_k^{(\frac{n-2}{2}, \frac{n-2}{2})}$.

\noindent 

\begin{lemma} \label{zeros-P_kn}
	For every positive integer $k \geq 1$, the polynomial  $P_{k, n} $ has  $k$ zeroes 	simples in $\left(-1, 1\right)$.
\end{lemma}

\begin{proof}
	For $k= 1, 2$, using the formula \ref{P_kn-formula} we obtain that the polynomials $P_{1, n} $ and $P_{2, n} $ are given by 
	
	$$P_{1, n}(t) = t, \;\; P_{2, n}(t)= \frac{n+1}{n} t^2 -\frac{1}{n}$$  	
	and the lemma follows for $k= 1, 2$. 
	
	Suppose that it is also true for $k$ and let's prove for $k+1$. Let 
	$t_1, t_2, \ldots t_k$ be the zeroes of  $P_{k, n} $ in $(-1, 1)$. Since 
	$P_{k, n} $ and $P_{k+1, n} $ are solutions of equation  \ref{eq-eigenvalues} we have 
	
	$$
	(1-t^2)P_{k, n}''(t)-nt P_{k, n}'(t) +  \frac{\lambda (q-2)}{1+\delta^{-1}} P_{k, n}(t)=0
	$$
	
	$$
	(1-t^2)P_{k+1, n}''(t)-nt P_{k+1, n}'(t) +  \frac{\lambda (q-2)}{1+\delta^{-1}} P_{k+1, n}(t)=0
	$$
	Note that $(k+1)(n+k) > k(n+k-1)$ and  
	$$\frac{P_{k+1, n}'(-1)}{P_{k+1, n}(-1)} = \frac{-(k+1)(n+k)}{n} < 
	\frac{-k(n+k-1)}{n}  = \frac{P_{k, n}'(-1)}{P_{k, n}(-1)}$$

	then by Sturm's comparison theorem, between every two zeroes of  $P_{k, n}$ there are at least one zero of $P_{k+1, n}$ and the i-th zero of $P_{k+1, n}$ is less than the i-th zero of $P_{k,  n}$. So $P_{k+1, n}$ has at least one zero in  each interval $(-1, t_1), (t_1, t_2), \ldots , (t_{k-1}, t_k)$. 
	Now, for every positive integer  $k \geq 1$ we have that $P_{k, n}(1) =1$ (See \cite{Morimoto}), and if $k$ is even, $P_{k, n}$ is an even function. If $k$ is odd  $P_{k, n}$ is an odd function.  It follows that 
	$P_{k, n}(-1) = (-1)^k$. 
	Let
	$\tilde{t}_1,  \tilde{t}_2, \ldots \tilde{t}_k$ be the zeroes of  $P_{k+1, n}.$ Then $P_{k, n}'(t_i)$ and $P_{k+1, n}'(\tilde{t}_i)$ has oposite signs. Since $t_k$ is the last zero of $P_{k, n}$  and $P_{k, n}(1) =1$, we have that  $P_{k, n}'(t_k)$  is positive  and  therefore $P_{k+1, n}'(\tilde{t_k})$ must be  negative.  But  $P_{k+1, n}(1) =1$, this implies that  necessarily  $P_{k+1, n}$ must change sign in $(\tilde{t_k}, 1)$, therefore $P_{k+1, n}$ has $k+1$ zeroes in $(-1, 1) $.
\end{proof}

\begin{lemma}
	For any $u, v\in  C^2[-1, 1]$ 
\begin{equation}
	\int_{-1}^1 F_w'(0, \lambda)[v] u \;  (1-t^2)^{\frac{n-2}{2}}dt =  	\int_{-1}^1F_w'(0, \lambda)[u] v \; (1-t^2)^{\frac{n-2}{2}} dt
\end{equation}	
	
\end{lemma}

\begin{proof}
Note that 

$$\left( (1-t^2)^{\frac{n-2}{2}} v ' \right)'   = -nt  (1-t^2)^{\frac{n-2}{2}}v' + (1-t^2)^{\frac{n}{2}} v''	$$

Then, using integration by parts 

\begin{eqnarray*}
	\int_{-1}^1 F_w'(0, \lambda)[v] u \;  (1-t^2)^{\frac{n-2}{2}}dt &=& 	
	\int_{-1}^1 \left( (1-t^2) v'' -nt v' +\frac{\lambda(q-2)}{1+\delta^{-1}} v  \right) u\;   (1-t^2)^{\frac{n-2}{2}}dt  \\
	&=& 
	\int_{-1}^1 (1-t^2)^{\frac{n}{2}} v'' u -nt (1-t^2)^{\frac{n-2}{2}} v' u+ \frac{\lambda(q-2)}{1+\delta^{-1}} v   u\;   (1-t^2)^{\frac{n-2}{2}}dt  \\
	&=&  
	\int_{-1}^1 \left( (1-t^2)^{\frac{n-2}{2}} v ' \right)' u + \frac{\lambda(q-2)}{1+\delta^{-1}} v   u\;   (1-t^2)^{\frac{n-2}{2}}dt\\
	&=& 
	\int_{-1}^1 \left( (1-t^2)^{\frac{n-2}{2}} u' \right)' v  + \frac{\lambda(q-2)}{1+\delta^{-1}}   u v\;    (1-t^2)^{\frac{n-2}{2}}dt\\
	&=& \int_{-1}^1  (1-t^2)u'' -2nt u' +\frac{\lambda(q-2)}{1+\delta^{-1}}   uv\;    (1-t^2)^{\frac{n-2}{2}}dt\\
	&=&	\int_{-1}^1 F_w'(0, \lambda)[u] v \;  (1-t^2)^{\frac{n-2}{2}}dt
\end{eqnarray*}
\end{proof}

It folowos from the previous lemma, that  the operator $ F_w'(0, \lambda_k) $ is self adjoint in the $(1-t^2)^{\frac{n-2}{2}} $-weighted $L^2 $ product on $L^2[-1, 1]$.

\begin{Theorem} \label{Bifurc-point}
	Let $\lambda_k = \frac{k(k+n-1)}{q-2}\left(1+ \frac{1}{\delta}\right)$. For every $k \geq 1$, the point $\left(0, \lambda_k\right)$ is a bifurcation point of the equation $F(w, \lambda) = 0$.
\end{Theorem}

\begin{proof}
	Consider the linearized operator at $(0, \lambda_k)$, denoted by $F_w'(0, \lambda_k)$.
	
	Recall that the eigenfunctions of the Laplace–Beltrami operator on the unit sphere $\mathbb{S}^n$ are the spherical harmonics. When restricted to zonal functions, these correspond to the ultraspherical polynomials $P_{k,n}(t)$, with associated eigenvalues $\mu_k = k(k+n-1)$ for $k \in \mathbb{N}$. These polynomials form an orthogonal basis of $L^2([-1,1], (1 - t^2)^{\frac{n-2}{2}} dt)$.
	
	The operator $F_w'(0, \lambda_k)$ is not invertible. Indeed, for any $u \in C^{2,\alpha}([-1,1])$, we have
	\[
	0 = \int_{-1}^1 F_w'(0, \lambda_k)[P_{k,n}(t)] \cdot u(t)\, (1 - t^2)^{\frac{n-2}{2}} dt = \int_{-1}^1 F_w'(0, \lambda_k)[u(t)] \cdot P_{k,n}(t)\, (1 - t^2)^{\frac{n-2}{2}} dt,
	\]
	which implies that the range of $F_w'(0, \lambda_k)$ satisfies
	\[
	\mathrm{Range}(F_w'(0, \lambda_k)) = \left\{ y \in C^{0,\alpha}([-1,1]) : \int_{-1}^1 y(t)\, P_{k,n}(t)\, (1 - t^2)^{\frac{n-2}{2}} dt = 0 \right\}.
	\]
	
	Next, we compute the mixed derivative of $F$ at $(0, \lambda_k)$ in the direction of $P_{k,n}$. We obtain
	\[
	F_{w,\lambda}''(0, \lambda_k)[P_{k,n}] = (q-2) P_{k,n}(t).
	\]
	Since $\int_{-1}^1 P_{k,n}(t)\, dt \neq 0$, it follows that
	\[
	F_{w,\lambda}''(0, \lambda_k)[P_{k,n}] \notin \mathrm{Range}(F_w'(0, \lambda_k)).
	\]
	
	Therefore, all the hypotheses of the Crandall–Rabinowitz bifurcation theorem (see \cite[Theorem 2.8]{Ambrosetti-Malchiodi}) are satisfied, and we conclude that $(0, \lambda_k)$ is a bifurcation point of $F(w, \lambda) = 0$.
	
	This completes the proof.
\end{proof}

This bifurcation gives rise to a nontrivial branch of solutions emerging from the trivial line at $\lambda = \lambda_k$. More precisely, near $(0, \lambda_k)$, the solution set is the union of the trivial branch $\lambda \mapsto (0, \lambda)$ and a smooth curve of nontrivial solutions $(w(s), \lambda(s))$, satisfying $w(s) \not\equiv 0$ for $s \neq 0$.

\begin{equation}\label{psi}
w(s) = s P_{k, n} + \psi (s), \;\;\;   \lambda(0) = \lambda_k, \;\;\;\psi (0)= 0 = \psi' (0)
\end{equation}

It follows that 
	\[
\frac{\partial w}{\partial s}(0) = P_{k, n} ,  \quad  \lambda(0) = \lambda_k, \quad  \quad w(0) = 0, \quad 
\]

\begin{lemma}
	Let $k$ be any positive integer.
	\begin{enumerate}
		\item If $k$ is odd, then 
		\( \displaystyle 
		\int_{-1}^1 P_{k,n}^3(t) (1 - t^2)^{\frac{n-2}{2}}\, dt = 0.
		\)
		\item If $k$ is even, then 
			\( \displaystyle  
		\int_{-1}^1 P_{k,n}^3(t) (1 - t^2)^{\frac{n-2}{2}}\, dt > 0.
		\)
	\end{enumerate}
\end{lemma}

\begin{proof}\hfill
	
	1.  If $k$ is odd, then the Gegenbauer polynomial $P_{k,n}(t)$ is an odd function. Since $(1 - t^2)^{\frac{n-2}{2}}$ is even, the product $P_{k,n}^3(t)(1 - t^2)^{\frac{n-2}{2}}$ is an odd function. Thus, the integral over the symmetric interval $[-1,1]$ vanishes.
	
	\vspace{1em}
	2.  
	Following the approach of Gasper~\cite{Gasper}, we expand the cubic power of $P_{k,n}$ in the orthogonal basis $\{P_{j,n}(t)\}$:
	\[
	P_{k,n}^3(t) = \sum_{j=0}^{3k} G_j\, P_{j,n}(t),
	\]
	where the coefficients $G_j$ are the linearization coefficients of the product $P_{k,n}^2(t) \cdot P_{k,n}(t)$. To analyze these, we consider the normalized Gegenbauer polynomials
	\[
	R_{k,n}(t) = \frac{P_{k,n}(t)}{P_{k,n}(1)},
	\]
	so that $R_{k,n}(1) = 1$. The square of the normalized polynomial has the expansion
	\[
	R_{k,n}^2(t) = \sum_{i=0}^{2k} G_i\, R_{i,n}(t).
	\]
	
	According to Hylleraas~\cite{Hylleraas}, the linearization coefficients satisfy
	\[
	G_i = (-1)^i c_i, \quad \text{with } c_i > 0.
	\]
	Thus, $G_0 = c_0 > 0$. To guarantee this positivity, Gasper~\cite[formula (5)]{Gasper} introduces a recurrence relation for auxiliary coefficients $d_j$ of the form:
	\begin{equation}\label{eq:gasper-recurrence}
		A_j\, d_{j+1} = B_j\, d_j - C_j\, d_{j-1},
	\end{equation}
	where in the ultraspherical case ($\alpha = \beta = \frac{n-2}{2}$, so $a = n-1$, $b = 0$, and $s = 0$), the coefficients are given by:
	\begin{align*}
		A_j &= (j+1)(2j + n)(2k + j + 2(n-1))(2k - j), \\
		B_j &= j(2k + j - 1 + 2(n - 1))(2k - j + 1)(2j), \\
		C_j &= (j + n - 2)(2j + n - 2)(2j + n - 1)(2k + j - 1 + 2(n - 1))(2k - j + 1).
	\end{align*}
	
	Gasper~\cite[p.~174]{Gasper} explicitly states that $d_0 > 0$ and $d_{2k} > 0$, providing the base for an inductive proof. The coefficients satisfy $A_j > 0$ and $C_j > 0$ for all $j$. Although $B_j$ may change sign, the recurrence structure guarantees that if $d_j > 0$ and $d_{j-1} > 0$, then $d_{j+1} > 0$. This propagates the positivity of all $d_j$.
	
	To analyze the possible sign change of $B_j$, Gasper introduces the auxiliary polynomial
	\[
	\begin{aligned}
		Q_J &= (J+2)^2(J+2k+2n-1)(2k-J-1)(2J+n) \\
		&\quad - (J+1)^2(J+2k+2n-2)(2k-J)(2J+n+2),
	\end{aligned}
	\]
	with $J = j - 1$. This polynomial arises as the difference $Q_J = A_{J+1} - A_J$, where $A_j$ is the coefficient of $d_{j+1}$ in the recurrence~\eqref{eq:gasper-recurrence}. Gasper shows that $Q_J$ is a polynomial of degree 4 in $J$ which changes sign exactly once on the interval $(0, \infty)$~\cite[p.~175]{Gasper}. This result ensures that the sequence of coefficients $A_j$ does not oscillate, and hence the recurrence is stable with respect to sign propagation. As a result, the positivity of $d_0$ and $d_1$ is sufficient to conclude that $d_j > 0$ for all $j$, and in particular that $G_0 > 0$.
	
	Now, returning to the integral,
	\[
	\int_{-1}^1 P_{k,n}^3(t)\, (1 - t^2)^{\frac{n-2}{2}}\, dt = \sum_{j=0}^{3k} G_j \int_{-1}^1 P_{j,n}(t)\, (1 - t^2)^{\frac{n-2}{2}}\, dt.
	\]
	By the orthogonality of the Gegenbauer polynomials, all terms vanish except when \( j = 0 \), for which \( P_{0,n}(t) = 1 \). Hence,
	\[
	\int_{-1}^1 P_{k,n}^3(t)\, (1 - t^2)^{\frac{n-2}{2}}\, dt = G_0 \int_{-1}^1 (1 - t^2)^{\frac{n-2}{2}}\, dt.
	\]
	Since $G_0 > 0$ and the integral is finite and positive, we conclude that
	\[
	\int_{-1}^1 P_{k,n}^3(t)\, (1 - t^2)^{\frac{n-2}{2}}\, dt > 0. \qedhere
	\]
\end{proof}


\begin{lemma} \label{lambda-prima}
	Let \( k \) be an even integer, and let \( ( w(s), \lambda(s) ) \) be the branch of nontrivial solutions to \( F(w, \lambda) = 0 \) that bifurcates from the trivial solution at \( (0, \lambda_k) \). Then, the derivative of \( \lambda \) with respect to \( s \) at \( s = 0 \) satisfies
	\[
	\frac{d\lambda}{ds}(0) < 0.
	\]
\end{lemma}
We will use $'$ to denote differentiation with respect to the variable $t \in [-1,1]$. We have
\[
(1 - t^2)w''(t) - n t w'(t) + \frac{\lambda (q - 2)}{1+\delta^{-1}}\left[(w(t)+1)^{q-1}-w(t)-1\right] = 0.
\]

	Differentiating with respect to $s$, we have:
	
	\[	(1 - t^2)\frac{dw}{ds}''(t) - n t \frac{dw}{ds}'(t)  + \frac{\lambda(q - 2)}{1+\delta^{-1}}\left[(q-1)(w+1)^{q-2}\frac{dw}{ds}(t)  		 -\frac{dw}{ds}(t)\right]   \]
		\[ + \frac{d\lambda}{ds}\frac{(q-2)}{1+\delta^{-1}}\left[(w+1)^{q-1}-w-1\right]=0. \]

		Differentiating again with respect to $s$ and evaluating at $s=0$, we obtain
		\[
			(1 - t^2)\frac{d^2 w}{ds^2}''(0) - n t \frac{d^2 w}{ds^2}'(0) 
			+ \frac{\lambda_k(q - 2)}{1+\delta^{-1}}\left[(q-1)(w+1)^{q-2}\frac{dw}{ds}(0)-\frac{dw}{ds}(0)\right]  \]
			\[ - 2\frac{d\lambda}{ds}(0)\frac{(q-2)}{1+\delta^{-1}}\left[(w+1)^{q-1}-w-1\right]=0.
		\]
		
		Multiplying by $P_{k,n}(t)(1-t^2)^{\frac{n-2}{2}}$ and integrating over $[-1,1]$, we obtain
		\begin{align*}
			&\int_{-1}^{1}\left[(1 - t^2)\frac{d^2 w}{ds^2}''(0) - n t \frac{d^2 w}{ds^2}'(0)\right]P_{k,n}(t)(1-t^2)^{\frac{n-2}{2}}dt\\[5pt]
			&+\frac{\lambda_k(q - 2)(q-2)}{1+\delta^{-1}}\int_{-1}^{1}\frac{d^2 w}{ds^2}(0)P_{k,n}(t)(1-t^2)^{\frac{n-2}{2}}dt\\[5pt]
			&-2\frac{d\lambda}{ds}(0)\frac{(q-2)}{1+\delta^{-1}}\int_{-1}^{1}\left[(w+1)^{q-1}-w-1\right]P_{k,n}(t)(1-t^2)^{\frac{n-2}{2}}dt=0.
		\end{align*}
		
		Thus, solving  for $	\frac{d\lambda}{ds}(0)$  we get
		\[
		\frac{d\lambda}{ds}(0) =- \frac{\lambda_k (q-1) \int_{-1}^{1} P_{k,n}^3(t)(1-t^2)^{\frac{n-2}{2}}dt}{2\int_{-1}^{1}P_{k,n}^2(t)(1-t^2)^{\frac{n-2}{2}}dt}.
		\]
	
	By the previous lemma, the integral in the numerator is positive when  	$k $  is even and zero when 
	$k$ is odd, from which the result follows.


\begin{Theorem}\label{P1} 
	Let \( f \) be a proper isoparametric function on the Riemannian manifold \((M, g)\). If \( p < q_f \), then there exists \( \lambda_0 > 0 \) such that, for every \( \lambda \in (0, \lambda_0) \), any positive \( f \)-invariant solution \( u \) of equation \textup{(1.1)} satisfies \( u \equiv 1 \).
\end{Theorem}

\begin{Theorem} \label{P2}
	Let \(1 < q < q_f\). Then for any positive constants \(0 < \varepsilon < M\), the space of positive \(f\)-invariant solutions of Equation~\emph{(1.1)} with \(\lambda \in [\varepsilon, M]\) is compact.
\end{Theorem}

	\section{Proof of Theorem~\ref{t2}}
	
	Let us consider the space \( D \) of nontrivial positive \( f \)-invariant solutions of equation~\eqref{eq-w+1}, defined by
	\[
	D := \left\{ (w, \lambda) \in \left( C^{2,\alpha}([-1,1]) \setminus \{0\} \right) \times (0,\infty) : w \text{ is a positive solution of } \eqref{eq-w+1} \right\}.
	\]
	
	By Theorem \ref{Bifurc-point}, each point \( (0, \lambda_k) \) is a bifurcation point of the equation \( F(w,\lambda) = 0 \). Near such a point, the branch of nontrivial solutions can be locally parameterized by a smooth curve \( s \mapsto (w(s), \lambda(s)) \), with
	\[
	w(s) = s P_{k,n} + o(s^2), \quad \lambda(0) = \lambda_k, \quad s \to 0.
	\]
	
	Since \( \frac{dw}{ds}(0) = P_{k,n} \), and  by lemma \ref{zeros-P_kn} \( P_{k,n} \)  has exactly \( k \) simple zeros in the open interval \( (-1,1) \). Then, by continuity of roots of smooth functions under perturbation in \( C^2([-1,1]) \), it follows that for \( s \neq 0 \) sufficiently small, the function \( w(s) \) also has exactly \( k \) simple zeros in \( (-1,1) \).
	Therefore, for small \( s \neq 0 \), the solution \( w(s) \) attains the value 0 exactly \( k \) times in \( (-1,1) \).
	
	We now define the connected component 	\(  	D_k^+ \)  of  \(D\)  containing  the set \(   \{ (w(s), \lambda(s)) : s > 0 \} \)  and the connected component    \(  D_k^- \) containing  the set \(   \{ (w(s), \lambda(s)) : s <  0 \} \).  
	
	 Let \( D_k := D_k^+ \cup D_k^- \). 	Since \( P_{k,n}(1) > 0 \), we have \( w(s)(1) > 0 \) for \( s > 0 \) small. Furthermore, if \( w(s)(0) = 0 \) for some \( s > 0 \), then by uniqueness of solutions to analytic ODEs and initial conditions, \( w(s) \equiv 0 \), contradicting the assumption that \( w(s) \neq 0 \). Hence, \( w(s) \) is strictly nonzero near the boundary, and the sign of \( w(-1) \) is preserved in a neighborhood.
	
	Let \( (w, \lambda) \in D \) be any solution. At each zero \( t_0 \) of \( w(t) \), we must have \( w'(t_0) = 0 \), due to the boundary conditions~\eqref{boundary-cond-1}. Moreover, if \( w(-1) > 0 \) (or \( < 0 \)), then by continuity of the evaluation map in the \( C^{2,\alpha} \)-topology, there exists a neighborhood \( U \subset C^{2,\alpha}([-1,1]) \times \mathbb{R} \) such that every \( (v, \mu) \in U \) satisfies \( v(-1) > 0 \) (or \( < 0 \)), and the number of zeros of \( v \) equals that of \( w \).
	
	Consequently, both the number of preimages of zero and the sign of \( w(-1) \) are constant on each connected component of \( D \). In particular, for every \( (w,\lambda) \in D_k \), the number of points \( t \in (-1,1) \) such that \( w(t) = 0 \) is exactly \( k \).
	
	From this, we deduce that if \( (0,\lambda) \in D_k \), then necessarily \( \lambda = \lambda_k \), as only the eigenfunction \( P_{k,n} \) has exactly \( k \) simple zeros in \( (-1,1) \).
	
	By Lemma~\ref{lambda-prima}, we have \( \frac{d\lambda}{ds}(0) < 0 \), so for \( s > 0 \) sufficiently small,
	\[
	\lambda(s) < \lambda_k.
	\]
	
	Now, by Theorem~\ref{P1}, there exists \( \lambda_0 > 0 \) such that if \( (w,\lambda) \in D_k^+ \), then necessarily \( \lambda \geq \lambda_0 \). From Theorem~\ref{P2}, the set
	\[
	\{ (w, \lambda) \in D_k : \lambda \in [\lambda_0, \lambda_k] \}
	\]
	is compact in \( C^{2,\alpha}([-1,1]) \times \mathbb{R} \). Therefore, there exists \( (w^*, \lambda^*) \in D_k^+ \) such that \( \lambda^* \) is minimal among all values of \( \lambda \) in \( D_k^+ \).
	
	Suppose \( w^* \) were constant. Then \( (w^*, \lambda^*) = (0, \lambda^*) \) and would lie on the trivial branch, implying \( \lambda^* = \lambda_i \) for some \( i < k \), contradicting the fact that all elements in \( D_k \) have \( k \) zeros. Thus, \( w^* \) is non-constant, and \( (w^*, \lambda^*) \in D_k^+ \).
	
	Assume, towards a contradiction, that \( D_w F(w^*, \lambda^*) \) is an isomorphism. Then by the implicit function theorem, there exists a smooth local curve \( s \mapsto (w(s), \lambda(s)) \) of solutions to \( F(w(s), \lambda(s)) = 0 \), with \( w(0) = w^* \), \( \lambda(0) = \lambda^* \), and \( (w(s), \lambda(s)) \in D_k^+ \). Since \( \lambda^* \) is minimal, it follows that
	\[
	\left.\frac{d\lambda}{ds}\right|_{s=0} = 0.
	\]
	
	Differentiating the equation
	\[
	F(w(s), \lambda(s)) = -\Delta_{G_\delta} w + \lambda \left( w + 1 - (w + 1)^{q-1} \right) = 0
	\]
	with respect to \( s \) at \( s = 0 \), we obtain the linearized equation
	\[
	-\Delta v + \lambda^* (v - q u_*^{q - 2} v) = 0,
	\]
	where \( v = \left. \frac{dw}{ds} \right|_{s=0} \), and \( u_* = w^* + 1 \). This implies \( v \in \ker(D_w F(w^*, \lambda^*)) \), contradicting the assumption that the linearized operator is invertible.
	
	Therefore, \( D_w F(w^*, \lambda^*) \) is not an isomorphism, and \( (w^*, \lambda^*) \) is a degenerate solution of equation~\eqref{eq-w+1}. Consequently, \( u^* := w^* + 1 \) is a positive, \( f \)-invariant, non-constant solution of equation~\eqref{eq3}, which depends nontrivially on both factors.
	
	This completes the proof.
	\qed

\end{document}